\documentclass[a4paper,11pt]{amsart}
\usepackage{amsmath,amsthm,amssymb,tikz-cd}

\usepackage[bookmarks=false]{hyperref}
\numberwithin{equation}{section}

\theoremstyle{plain}

\newtheorem{theorem}{Theorem}[section]

\newtheorem{lemma}[theorem]{Lemma}
\newtheorem{proposition}[theorem]{Proposition}
\newtheorem{corollary}[theorem]{Corollary}
\theoremstyle{definition}
\newtheorem{definition}[theorem]{Definition}
\newtheorem*{definition*}{Definition}
\newtheorem{example}[theorem]{Example}

\begin{document}

\title[$\phi\,$-SECTIONAL CURVATURE OF STATISTICAL STRUCTURES \\ ON ALMOST CONTACT METRIC MANIFOLDS
]
{$\phi\,$-SECTIONAL CURVATURE OF STATISTICAL STRUCTURES \\ ON ALMOST CONTACT METRIC MANIFOLDS}

\author[Abbas Heydari, Sadegh Mohammadi]{Abbas Heydari, Sadegh Mohammadi }

\address{Department of Mathematics\\ Faculty of Mathematical Sciences\\ Tarbiat Modares University\\ Tehran 14115-134, Iran}
\email{aheydari@modares.ac.ir , mohammadisadegh67@gmail.com}

\thanks{}
%\keywords{Partial actions, ergodicity, Koopman partial representation, globalization, Kakutani skyscraper }
%\subjclass{Primary 37B05 ; Secondary 43A07, 57S05}

%\author[]{}
%\address{}
%\email{}
%\date{\today}
%\author[]{}
%\address{}
%\email{}

%\thanks{}
%\footnote{}\footnote{}
%\journal{Mathematical Notes}

\begin{abstract}
In this article, we study and analyze the $\phi$-sectional curvature induced by a statistical structure on an almost contact metric manifold. We demonstrate that this sectional curvature is always non-positive. Additionally, we present equivalent statements regarding the vanishing of this type of sectional curvature. Furthermore, we derive a sufficient condition for an almost contact statistical manifold to be classified as a cosymplectic statistical manifold. 
\end{abstract}

\maketitle

\section{introduction}
Statistical structures were  first introduced in the 1980s by S. L. Lauritzen in \cite{13}. Later, in the article \cite{12}, T. Kurose introduced an equivalent definition of statistical manifolds within the framework of affine differential geometry, which will be discussed in the following section.
Statistical manifolds provide a geometric framework for studying and modeling probability distributions and play a significant role in various fields of science such as neutral networks, machine learning, information geometry, and physics \cite{2,4,7,8,17}. For further reading in this area, Amari and Nagaoka's textbook \cite{3} and references \cite{9}, \cite{13}, and \cite{6} are recommended. Statistical manifolds play a fundamental role in information geometry and the modeling of statistical data, while almost contact manifolds are widely used in Riemannian geometry to describe dynamical and physical systems. Combining these two structures creates a powerful framework that bridges statistical analysis with geometric modeling, offering new insights into complex systems where both statistical and physical components interact. This synthesis not only enhances our understanding of such systems but also opens new avenues for studying nonlinear dynamics and predictive modeling. The concept of the almost contact statistical manifolds was introduced by F. Malek in the article \cite{1}.  
One of the main approaches to studying and identifying manifolds is through their curvature. Since sectional curvature plays a crucial role in understanding a manifold's geometric structure, its study is of great importance. Sectional curvature is a fundamental concept in differential geometry, primarily associated with Riemannian metrics. While the curvature tensor can be defined for any connection, sectional curvature— which assigns a numerical value to a tangent space plane—typically requires a metric tensor. A key observation is that if a tensor can be defined in relation to the Riemannian metric in the same way as the curvature tensor, then the concept of sectional curvature can be extended. This idea was successfully explored in the work of B. Opozda, who introduced statistical sectional curvature corresponding to any statistical structure on a manifold \cite{16}.  
In this paper, we investigate the sectional curvature associated with an almost contact statistical structure on a manifold and establish that this curvature is always non-positive. Additionally, we connect this concept to previous research on statistical structures in almost contact metric manifolds, emphasizing its relevance in understanding their geometric properties. Specifically, we show that the vanishing of this sectional curvature reveals essential characteristics of the statistical structure on an almost contact metric manifold. Finally, we demonstrate that if the statistical connection is compatible with the almost contact metric structure, then every almost contact statistical manifold must be a cosymplectic statistical manifold.  

\section{Statistical Manifolds}
Statistical structures on a manifold $M$ can be defined through three equivalent approaches. Firstly, a statistical manifold is a Riemannian manifold $(M,g)$ equipped with a symmetric 3-covariant tensor field \cite{13}.  For an alternative equivalent definition of statistical manifolds, consider the pair $(g,\nabla)$, where $g\in \otimes^0_2 M$ is a Riemannian metric tensor on $M$, and $\nabla$ is a torsion-free affine connection such that the associated tensor field $\nabla g\in\otimes ^0_3 M$ is symmetric. In this case, the pair $(g, \nabla)$ is called a statistical structure on $M$, and the triplet $(M, g, \nabla)$ is known as a statistical manifold \cite{12}. Every Riemannian manifold $(M, g)$, along with the associated Levi-Civita connection, is a typical example of statistical manifolds. In other words, statistical manifolds can be considered as an extension of Riemannian manifolds.  An alternative but equivalent formulation of statistical manifolds states that the triplet $(M, g, K)$ defines a statistical manifold, where $g\in\otimes^0_2 M$ is a Riemannian metric tensor and $K\in S^2 (M;TM)$ is a symmetric tensor field on $M$ that is symmetric with respect to $g$.  The symmetry of the tensor field $K$ with respect to the Riemannian metric tensor $g$ means that the cubic form $C(X,Y,Z) := g(X,K(Y,Z))$ is symmetric.  The validity of this equivalence is established by jointly considering the mathematical fact that the set of all torsion-free affine connections on the manifold $M$ forms an affine space whose associated vector space is $S^2(M;TM)$, along with the existence of musical isomorphisms in the Riemannian manifold $(M, g)$.  Furthermore, the statistical connection $\nabla$, the Levi-Civita connection $\nabla^\circ$, the tensor field $K$, and the Riemannian metric $g$ satisfy the relationships $K=\nabla -\nabla^\circ$ and $g(X,K(Y,Z))=-2(\nabla _{X} g)(Y,Z)$.  

\section{Almost Contact Metric Manifolds}
This section is based on \cite{5} and has been written accordingly.
\begin{definition} Let $M$ be an odd-dimensional manifold. The triplet $(\phi,\xi,\eta)$, where $\phi$ is a tensor field of type (1,1), $\xi$ is a vector field, and $\eta$ is a 1-form on $M$, is called an almost contact structure on $M$, provided that 
\[ 
	\phi^2=-I+\eta\otimes\xi\quad , \quad \eta(\xi)=1
\]
 are satisfied.  
\end{definition}
If $(\phi,\xi,\eta)$ is an almost contact structure on the manifold $M^{2n+1}$, then 
\begin{align*}
	\phi (\xi)=0\quad , \quad \eta \circ \phi =0 \quad , \quad rank(\phi)=2n
\end{align*}
 hold true.
\begin{definition}
	Suppose $(\phi,\xi,\eta)$ is an almost contact structure on the manifold $M^{2n+1}$. If there exists a Riemannian metric $g\in\otimes ^0_2 M$ such that 
\begin{align*}
	\forall X,Y\in \mathfrak{X}(M);g(\phi X, \phi Y)=g(X,Y)-\eta (X) \eta (Y)
\end{align*} 
holds, then the quadruple $(\phi,\xi,\eta,g)$ defines an almost contact metric manifold.  
\end{definition} 

It is straightforward to see that if $(M,\phi,\xi,\eta ,g)$ is an almost contact metric manifold, then 
\begin{align*}
	g(\xi,\xi)=1\quad ,\quad \eta (X)=g(X,\xi)\quad ,\quad g(\phi X,Y)+g(X,\phi Y)=0.
\end{align*}  
\begin{definition}
An almost contact metric manifold $(M,\phi,\xi,\eta ,g)$ is called a cosymplectic manifold if $\nabla^\circ \phi = 0$, where $\nabla^\circ$ is the unique Levi-Civita connection associated with the Riemannian manifold $(M,g)$.  
\end{definition}
\section{Almost Contact Statistical Manifolds}
\begin{definition} \label{4.1}
	 Suppose $(M,\phi,\xi,\eta ,g)$ is an almost contact metric manifold, and $(g,\nabla = \nabla^\circ+ K)$ is a statistical structure on $M$. Then, $(M,\phi,\xi,\eta ,g,\nabla=\nabla^\circ+K)$ is called an almost contact statistical manifold if 
	\begin{align*}
		\forall X,Y\in \mathfrak{X}(M); K(X,\phi Y)+\phi K(X,Y)=0
	\end{align*}
	 is satisfied \cite{1}.  
\end{definition}
\begin{proposition}
 If $(M, \phi ,\xi , \eta , g)$ is an almost contact metric manifold and $(g, \nabla = \nabla^\circ + K)$ is a statistical structure on $M$, then $(M, \phi ,\xi ,\eta , g,\nabla =\nabla^\circ +K)$ is an almost contact statistical manifold if, and only if, 
 \begin{align*}
 	\forall X,Y\in \mathfrak{X}(M); K(X,\phi Y)=K(\phi X,Y).
 \end{align*}

\end{proposition}

\begin{proof}
If  $(M, \phi ,\xi ,\eta , g,\nabla =\nabla^\circ +K)$ is an almost contact statistical manifold, then 
\begin{align*}
	K(X, \phi Y) = -\phi K(Y, X) = K(\phi X, Y)
\end{align*}
 holds.  
Conversely, if the statistical structure $(g, \nabla = \nabla^\circ+ K)$ on the almost contact metric manifold $(M, \phi ,\xi ,\eta , g)$ is such that $K(X,\phi Y)=K(\phi X,Y) $ is satisfied, then 
\begin{align*}
	g(K(\phi X, Y), Z) &= g(K(\phi X, Z), Y)=g(K(X, \phi Z), Y)\\
	&= g(K(X, Y), \phi Z) = g(-\phi K(X, Y), Z).
\end{align*}
Now, based on the nondegeneracy of the Riemannian metric tensor $g$, we conclude that $K(\phi X, Y) = -\phi K(X, Y)$ holds. Thus, the statistical structure $(g,\nabla =\nabla^\circ + K)$ on $(M,\phi ,\xi ,\eta , g)$ defines an almost contact statistical structure.
\end{proof}
\begin{lemma} \cite{14} \label{4.3} 
Assume that  $ (M, \phi, \xi, \eta, g, \nabla=\nabla^\circ + K) $ is an almost contact statistical manifold. Therefore, we have $K(X, \xi) = \lambda \eta(X) \xi$, where $\lambda = g(K(\xi, \xi), \xi) \in C^{\infty} (M)$.
\end{lemma}
\begin{corollary}
	If $(M, \phi, \xi, \eta, g, \nabla=\nabla^\circ + K)$ is an almost contact statistical manifold, then:  
\begin{align*}	
	K(\xi, \xi) = \lambda \xi , \quad K(\xi, \xi) = 0 \, \Leftrightarrow \, \lambda = 0.  
\end{align*}	 
\end{corollary} 
Here, we examine some simple properties of almost contact statistical manifolds, which will be useful in the following discussions. Note that some of these properties are independent of the statistical structure defined on the almost contact metric manifold.  

\begin{lemma} \label{4.5}  
Suppose  $(M, \phi, \xi, \eta, g, \nabla=\nabla^\circ + K)$ is an almost contact statistical manifold. Then, the following hold:  
\begin{itemize}
	\item \( g(\phi X, \xi) = 0 \),$\quad$ $\bullet$ \( K(\phi X, \xi) = 0 \),$\quad\bullet$ \( g(X, \phi X) = 0 \),
	\item \( g(\phi^2 X, Y) = g(X, \phi^2 Y) = -g(\phi X, \phi Y) \),
	\item \( K(\phi^2 X, Y) = K(X, \phi^2 Y) = K(\phi X, \phi Y) = \phi^2 K(X, Y) \),
	\item \( \forall X, Y \in \mathfrak{X}(M) ; K(X, \phi Y) = 0 \iff \phi K(X, Y) = 0 \),	
	\item \( \forall X \in \mathfrak{X}(M) ; \phi X = 0 \iff X\parallel\xi  \),
	\item \( \forall X, Y \in \mathfrak{X}(M) ; K(X, Y) = 0 \iff K(X, X) = 0 \),
	\item \( \forall X, Y \in \mathfrak{X}(M), X \perp \xi, Y\perp \xi ; K(X, Y) = 0 \iff K(X, X)=0\),
	\item \( \forall X, Y \in \mathfrak{X}(M) ; K(X, \phi Y) = 0 \iff K(X, \phi X) = 0 \),
	\item \( \forall X, Y \in \mathfrak{X}(M) ; \phi K(X, Y) = 0 \iff \phi K(X, X) = 0 \),	
	\item \( \forall X, Y \in \mathfrak{X}(M) ; \phi K(X, Y) = 0 \iff K(X, Y) \parallel \xi \),
	\item \( \forall X, Y \in \mathfrak{X}(M) ; K(X, Y) \parallel \xi \iff K(X, X) \parallel \xi \).
\end{itemize}	
\end{lemma}
\section{$\phi$-Sectional $K$-Curvature On Almost Contact Statistical Manifolds}
The concept discussed in this section has been informed by the definition provided in \cite{16}. An important topic that we will discuss in this paper is the concept of $\phi$-sectional $K$-curvature. We know that for any Riemannian manifold $(M, g)$, the notion of sectional curvature is defined. If $\nabla^\circ$ is the Levi-Civita connection of the Riemannian metric $g$, and $R^{\circ}$ is the curvature-like tensor dependent on the Levi-Civita connection $\nabla^\circ$, then the sectional curvature of the two-dimensional plane $\pi$ in the tangent space $T_p M$ is defined as
\begin{align}
\mathcal{K}^\circ(\pi) := \frac{g(R^\circ(X,Y)Y,X)}{Q(X,Y)},
\end{align} 
where $Q(X,Y) = g(X, X) g(Y,Y) - g(X, Y)^2$ and $X$, $Y$ are two linearly independent vectors in the plane $\pi$. 
Now, assume that  $(M, \phi, \xi, \eta, g, \nabla =\nabla^\circ +K)$  is an almost contact statistical manifold.  
In this case, with the help of the tensor field $K$, we can construct a curvature-like tensor on $M$ and use it to define the concept of $\phi$-sectional $K$-curvature for the almost contact statistical manifold  $(M, \phi, \xi, \eta, g, \nabla =\nabla^\circ +K)$.  
For this purpose, consider a special case where $V$ is an $n$-dimensional vector space equipped with a positive definite scalar product $g$. Let $K$ be a symmetric tensor of type $(1,2)$ on $V$ that is symmetric with respect to $g$, i.e., $g(X, K(Y, Z)) = g(Y, K(X, Z))$.  
This allows us to define an operator $K_X$, given by 
$K_X(Y):= K(X,Y)$ which produces a symmetric tensor with respect to $g$; $g(X,K_Z(Y))=g(Y,K_Z(X))$.
On the other hand, $K$ defines a symmetric cubic form on V as $C(X,Y,Z):=g(X,K(Y,Z))$ and a tensor 
$[K,K]$ as
\begin{align*}
[K,K](X,Y)Z:=[K_X,K_Y]Z:=K_X K_Y Z - K_Y K_X Z.
\end{align*}
Considering that $[K, K]$ possesses the following properties:
\begin{align*} 
	&[K,K](X,Y)Z+[K,K](Y,Z)X+[K,K](Z,X)Y=0,\\
	&g([K,K](X,Y)Z,W)=-g([K,K](Y,X)Z,W),\\
	&g([K,K](X,Y)Z,W)=-g([K,K](X,Y)W,Z),\\
	&g([K,K](X,Y)Z,W)=g([K,K](Z,W)X,Y),
\end{align*} 
it is indeed a curvature-like tensor on $V$, and the crucial aspect here is that we can define the sectional $K$-curvature of the two-dimensional plane $V$.
\begin{definition}
Assume that $(M, \phi, \xi, \eta, g, \nabla = \nabla^\circ + K)$  is an almost contact statistical manifold. A two-dimensional vector subspace $\pi$ of the tangent space $T_p M$ is called a $\phi$-section of $T_p M$ if there exists a vector $X \in T_p M$ that is orthogonal to $\xi$ such that $\{X, \phi X\}$ spans the vector subspace $\pi$.
The $\phi$-sectional $K$-curvature of $M$ with respect to $\pi$ is given by the formula
\begin{align*}
	\mathcal{K}_\phi (X) := \frac{g([K,K](X,\phi X)\phi X,X)}{Q(X,\phi X)},
	\end{align*}
where $Q(X, \phi X) = g(X, X) g(\phi X,\phi X) - g(X,\phi X)^2$. The value of $\mathcal{K}_\phi (X)$ 
is well-defined, as it independent of the specific choice of the linearly independent set
$\{X, \phi X\}$ from the plane $\pi$. If, for every $p \in M$, the $\phi$-sectional $K$-curvature of $M$ is constant with respect to all $\phi$-sections of the tangent space $T_p M$, then $M$ is a manifold with constant $\phi$-sectional $K$-curvature.
\end{definition}
We know that for every almost contact metric manifold $(M, \phi, \xi, \eta, g)$,  
there exists an orthonormal basis of the form  
$\{ e_1, \dots, e_n, \phi e_1, \dots, \phi e_n, \xi \} \subseteq T_p M$,  
which is called a $\phi$-basis \cite{5}. Therefore, any two-dimensional plane 
$\pi = span \{e_i,\phi e_i\}\leq T_p M$ is a $\phi$-section.\\
For any connection $\nabla$ on a Riemannian manifold $(M,g)$, one can define its conjugate connection
$\bar{\nabla}$ (with respect to $g$) by the following formula:
\begin{align*}
	g(\nabla_{X} Y,Z) + g(Y,\bar{\nabla}_{X} Z)=X\cdot g(Y,Z),
\end{align*}
for any vector fields $X$, $Y$, $Z$ on $M$.
If $(g, \nabla)$ is a statistical structure, then so is $(g, \bar{\nabla})$.
If $R$ is the curvature tensor of $\nabla$, and $\bar{R}$ is the curvature tensor of $\bar{\nabla}$, then for all $X$, $Y$, $Z$, $W$ we have $g(R(X,Y)Z,W)=-g(\bar{R}(X,Y)W,Z)$ \cite{15}.
Let $K$ be the difference tensor between $\nabla$ and the Levi-Civita connection $\nabla^{\circ}$, that is,
\begin{align*}
	\nabla_{X} Y=\nabla^{\circ}_{X} Y +K(X,Y),
\end{align*}

then we also have
\begin{align*}
		\bar{\nabla}_{X} Y=\nabla^{\circ}_{X} Y -K(X,Y).
\end{align*}

Now, if $(g, \nabla=\nabla^{\circ}+K)$ is a statistical structure on the manifold $M$, then the tensor
\begin{align*}
	S(X,Y)Z:=\frac{1}{2}\{R(X,Y)Z+\bar{R}(X,Y)Z\},
\end{align*}

called the statistical curvature tensor field, satisfies the following properties \cite{10}:
\begin{align*}
&S(X,Y)Z+S(Y,Z)X+S(Z,X)Y=0,\\
&g(S(X,Y)Z,W)=-g(S(Y,X)Z,W),\\
&g(S(X,Y)Z,W)=-g(S(X,Y)W,Z),\\
&g(S(X,Y)Z,W)=g(S(Z,W)X,Y).
\end{align*}

Therefore, the tensor $S$ allows us to define a curvature-like notion, called the statistical sectional curvature.
\begin{proposition} \cite{11}
If $(M,g, \nabla=\nabla^{\circ}+K)$ is a statistical manifold, then
	$S=R^{\circ}+[K,K]$,
where $R^{\circ}$ is the Riemannian curvature tensor field of the Riemannian manifold $(M, g)$.
\end{proposition}

Hence, on a statistical manifold $(M,g, \nabla=\nabla^{\circ}+K)$,
if $\mathcal{K}^S$, $\mathcal{K}^{\circ}$, and $\mathcal{K}$ denote the statistical sectional curvature, the Riemannian sectional curvature, and the sectional $K$-curvature, respectively, then the equality
\begin{align*}
	\mathcal{K} _{\phi}^{S}=\mathcal{K} _{\phi}^{\circ}+\mathcal{K} _{\phi},
\end{align*}
describes the relation among these types of sectional curvatures.
\begin{theorem} \label{5.3}
	Assume that $(M, \phi, \xi, \eta, g, \nabla = \nabla^\circ + K)$ is an almost contact statistical manifold. Then, the $\phi$-sectional $K$-curvature of the manifold $M$ is non-positive.
\end{theorem}
\begin{proof}
	Suppose that $\pi=span\{X,\phi X\}$ is a $\phi$-section of the tangent space $T_p M$.
	Now, to prove the theorem, we first compute the value of
	\begin{align*}
	g([K,K](X,\phi X)\phi X,X)&=
	g(K(X, K(\phi X, \phi X)), X) - g(K(\phi X, K(X, \phi X)), X) \\
	&=g(K(X, \phi^2 K(X, X)), X) - g(K(\phi X, -\phi K(X, X)), X) \\
	&=2 g(\phi^2 K(X, K(X, X)), X) \\
	&=-2 g(K(X, K(X, X)), X) + 2 \eta(K(X, K(X, X))) \eta(X) \\
	&=-2 g(K(X, X), K(X, X)) \\
	&= - 2 \|K(X, X)\|^2.
	\end{align*}
   Then, by obtaining the value of
	\begin{align*}
		Q(X,\phi X) &= g(X, X) g(\phi X, \phi X) - g(X, \phi X)^2 = g(X, X) g(\phi X, \phi X)\\
		&= g(X, X)(g(X, X) - \eta(X) \eta(X)) = g(X, X)^2 = \|X\|^4,
	\end{align*}
	we conclude that
	\begin{align*}
		\mathcal{K} _{\phi}(X)= -2 \frac{\|K(X, X)\|^2}{\|X\|^4} \leq 0.
	\end{align*}
\end{proof}
\begin{corollary} \label{5.4}
	Let $(M, \phi, \xi, \eta, g, \nabla = \nabla^\circ + K)$ be an almost contact statistical manifold. Then:
	\begin{itemize}
	\item \( 	\mathcal{K} _{\phi}^{S} \leq \mathcal{K} _{\phi}^{\circ} \),  
	\item If $\phi$-sectional $K$-curvature is non-negative, then \( \mathcal{K} _{\phi} = 0 \),  
	\item \( \mathcal{K} _{\phi} = 0 \) if and only if 
	\( \mathcal{K} _{\phi}^{S} = \mathcal{K} _{\phi}^{\circ} \), 
	\item \( \mathcal{K} _{\phi} = 0 \) if and only if for every vector field \( X \) orthogonal to the Reeb vector field \( \xi \), we have \( K(X,X) = 0 \). In other words,
	 \( \mathcal{K} _{\phi} = 0 \iff \forall X \in \mathfrak{X}(M), X \perp \xi ; \; K(X, X) = 0 \).
		\end{itemize} 
\end{corollary}
\begin{proposition} \label{5.5}
	If $(M, \phi, \xi, \eta, g, \nabla = \nabla^\circ + K)$ is an almost contact statistical manifold, then the following statements are equivalent:
	\begin{itemize}
		\item 	\( \mathcal{K} _{\phi} = 0 \),
		\item   \( K=\lambda \eta \otimes \eta \otimes \xi = \eta \otimes \eta \otimes K(\xi , \xi) \),
		\item   \( [K, K]=0\ \).	
	\end{itemize}
\end{proposition}
\begin{proof}
	Assume that $\mathcal{K} _{\phi} = 0$. It is known that any vector field $X$ on $M$ can be decomposed into its horizontal and vertical components, namely
	$X - \eta(X)\xi$ and $\eta(X) \xi$. Now, by applying Corollary \ref{5.4} and Lemma \ref{4.3}, one can conclude that:
	\begin{align*}
		K(X, X) =&
		K(X-\eta(X)\xi, Y-\eta(Y)\xi) + K(X-\eta(X)\xi, \eta(Y)\xi)\\
		&+K(\eta(X)\xi, Y-\eta(Y)\xi) + K(\eta(X)\xi, \eta(Y)\xi)\\
		=&\eta(Y)K(X-\eta(X)\xi, \xi) +
		\eta(X)K(\xi, Y-\eta(Y)\xi) +
		\eta(X)\eta(Y)K(\xi, \xi)\\
		=&\eta(Y)K(X,\xi) - \eta(X)\eta(Y)K(\xi,\xi)\\
		&+\eta(X)K(Y,\xi) - \eta(X)\eta(Y)K(\xi,\xi)\\
		&+\eta(X)\eta(Y)K(\xi,\xi)\\
		=&\lambda\eta(X)\eta(Y)\xi.
	\end{align*}
	Conversely, assuming the validity of the equality $K(X, X) = \lambda \eta(X) \eta(X) \xi =0$,
	and observing that for any vector field $X$ orthogonal $\xi$ to, we have $K(X,X) = 0$,
	it follows directly from Corollary \ref{5.4} that $\mathcal{K} _{\phi} = 0$.
    Finally, according to Lemma 3.5 from \cite{14}, the proof of the proposition is complete.   
\end{proof}
\begin{lemma} \label{5.6}
Assume that $(M, \phi, \xi, \eta, g, \nabla = \nabla^\circ + K)$ is an almost contact statistical manifold. Then, for any two vector fields $X$ and $Y$, the relation 
\begin{align*}
	(\nabla_X^\circ \phi)Y = (\nabla_X \phi)Y + 2\phi K(X, Y),
\end{align*}
is satisfied.
\end{lemma}
\begin{proof}
	Since $\nabla^{\circ} = \nabla - K$, it follows that:
	\begin{align*}
(\nabla^{\circ}_X \phi)Y &= \nabla^{\circ}_X \phi Y - \phi(\nabla^{\circ}_X Y)\\
&=(\nabla - K)_X \phi Y - \phi(\nabla_X Y - K(X,Y))\\
&= \nabla_X \phi Y - \phi \nabla_X Y - K(X,\phi Y) + \phi K(X,Y)\\
&=\nabla_X \phi Y + 2 \phi K(X,Y).
    \end{align*}	
\end{proof}
\begin{proposition}
	Assume that $(M, \phi, \xi, \eta, g, \nabla = \nabla^\circ + K)$ is an almost contact statistical manifold. Then, the following equivalence holds:
	\begin{align*}
		\mathcal{K} _{\phi} = 0 \iff \forall X,Y \in \mathfrak{X}(M) ; \; K(X, \phi Y) = 0.
	\end{align*}
\end{proposition}
\begin{proof}
	According to Proposition \ref{5.5} and Lemma \ref{5.6}, we have:
	\begin{align*}
		\mathcal{K} _{\phi} = 0
		&\Longrightarrow 
		(\nabla^{\circ}_X \phi)Y = (\nabla_X \phi)Y + 2\phi K(X,Y)\\
		&\Longrightarrow 
		(\nabla^{\circ}_X \phi)Y = (\nabla_X \phi)Y + 2\phi K(\lambda\eta(X)\eta(Y)\xi)\\
		&\Longrightarrow (\nabla^{\circ}_X \phi)Y = (\nabla_X \phi)Y\\
		&\Longrightarrow 
	     \nabla^{\circ}_X \phi Y - \phi \nabla^{\circ}_X Y =
	     \nabla_X \phi Y - \phi \nabla_X Y\\
	    &\Longrightarrow
	     \nabla_X \phi Y - \nabla^{\circ}_X \phi Y =
	     \phi(\nabla_X Y - \nabla^{\circ}_X Y)\\
	    &\Longrightarrow K(X, \phi Y) = \phi K(X,Y)\\
	    &\Longrightarrow K(X, \phi Y) =0.
\end{align*}
Conversely, suppose that $X$ is a vector field orthogonal to $\xi$.
In this case, we have:
\begin{align*}
	K(X, \phi X) =0 &\Longrightarrow \phi K(X, X) =0 \\
	&\Longrightarrow K(X, X) = g(K(X, X), \xi) \xi \\
	&\Longrightarrow K(X, X) = g(K(X, \xi), X) \xi \\
	&\Longrightarrow K(X, X) = \lambda \eta(X)g(X,\xi) \xi =0 \\
	&\Longrightarrow K(X, X) = \lambda \eta(X) \eta(X) \xi =0.
\end{align*}
Now, by applying Corollary \ref{5.4}, we conclude that the $\phi$-compatible $K$-curvature 
$\mathcal{K} _{\phi}$ is equal to zero.	
\end{proof}
From the propositions and results established above, we obtain the following theorem:
\begin{theorem} \label{5.8}
Let $(M, \phi, \xi, \eta, g, \nabla = \nabla^\circ + K)$ be an almost contact statistical manifold. Then the following statements are equivalent:
\begin{itemize}
\item \(\mathcal{K} _{\phi}=0\),
\item \( \mathcal{K} _{\phi}^{S} = \mathcal{K} _{\phi}^{\circ}\),
\item \( K=\lambda \eta \otimes \eta \otimes \xi = \eta \otimes \eta \otimes K(\xi , \xi) \),
\item \( [K, K]=0\ \),
\item \( S = R^{\circ}\),
\item \( \forall X \in \mathfrak{X}(M), X \perp \xi ; \; K(X, X) = 0\),
\item \( \forall X \in \mathfrak{X}(M) ; \; K(X, \phi X) = 0\),
\item \( \forall X \in \mathfrak{X}(M) ; \; \phi K(X, X) = 0\),
\item  \( \forall X \in \mathfrak{X}(M) ; \; K(X, X) \parallel \xi\).
\end{itemize}
\end{theorem}
\begin{corollary} \label{5.9}
	Let $(M, \phi, \xi, \eta, g, \nabla = \nabla^\circ + K)$ be an almost contact statistical manifold with $\mathcal{K} _{\phi}=0$.
	Then, if it further holds that $K(\xi,\xi) = 0$,
	it follows that $K = 0$.
\end{corollary}
\begin{theorem}
	Assume that $(M, \phi, \xi, \eta, g, \nabla = \nabla^\circ + K)$ is a contact statistical manifold such that $\mathcal{K} _{\phi}=0$.
	Then, the vector field $\xi$ is geodesic with respect to the statistical connection $\nabla$ (i.e., $\xi$ is $\nabla$-geodesic) if and only if $K = 0$.
	
\end{theorem}
\begin{proof}
	By Lemma 6.2 from \cite{5}, for every contact metric manifold $(M, \phi, \xi, \eta, g)$ we have
	\begin{align*}
		\nabla^{\circ}_\xi X = -\phi X - \phi h X,
	\end{align*}
	where $h = \frac{1}{2} \mathcal{L}_\xi \phi $. Since
	$h \xi = (\frac{1}{2} \mathcal{L}_\xi \phi)(\xi) = [\xi,\phi \xi] - \phi [\xi, \xi] = 0$,
	this immediately implies $\nabla^{\circ}_\xi \xi = 0$.
	
	Moreover, using the relation $K(\xi,\xi) = \nabla_\xi \xi = 0$,
	based on the Corollary \ref{5.9}, we clearly see that the vector field $\xi$ is geodesic with respect to the statistical connection $\nabla$ if and only if $K = 0$.
\end{proof}
\begin{theorem}\label{5.11}
	Assume that $(M, \phi, \xi, \eta, g, \nabla = \nabla^\circ + K)$ is a cosymplectic statistical manifold where the $\phi$-sectional $K$-curvature $\mathcal{K} _{\phi}$ is zero. Then,
	$\nabla_\xi \xi = 0$ if and only if $K = 0$.
\end{theorem}
\begin{proof}
	To establish the theorem, it suffices to verify $\nabla^{\circ}_\xi \xi = 0$, 
	based on the Corollary . Since, in every cosymplectic manifold, $\nabla^{\circ}\phi = 0$, 
	it follows that
	 $\phi \nabla^{\circ}_X \xi = \nabla^{\circ}_X \phi \xi = 0$. Thus, $\nabla^{\circ}_\xi$ and $\xi$ are parallel. On the other hand, given that $X.g(\xi,\xi) = 2g(\nabla^{\circ}_X \xi,\xi) = 0$, 
	we conclude that $\nabla^{\circ}_X \xi$ is orthogonal to $\xi$. Consequently, the vanishing of $\nabla^{\circ}_\xi \xi$ follows.
\end{proof}
\section{$\phi$-Compatible Almost Contact Statistical Structures}
In the study of almost contact statistical manifolds, it is natural to consider statistical connections that are compatible with the structure tensor $\phi$ . Specifically, we say that a statistical connection $\nabla$ is $\phi$-compatible if $\nabla \phi = 0$.

In this section, we investigate the geometric consequences of imposing 
$\nabla \phi=0$. Our main result shows that this condition forces the manifold to become cosymplectic in the statistical sense. That is, the requirement that be preserved by the statistical connection imposes significant restrictions on the geometry of the manifold.
\begin{proposition}
Let $(M, \phi, \xi, \eta, g, \nabla = \nabla^\circ + K)$ be an almost contact statistical manifold. Then, the following statements are equivalent:
\begin{itemize}
	\item $\nabla$ is a $\phi$-compatible statistical connection,
	\item $\forall X,Y\in \mathfrak{X}(M); \nabla _X\phi Y = \phi \nabla_X Y$,
	\item $\forall X,Y\in \mathfrak{X}(M); (\nabla ^{\circ}_X\phi)Y=2\phi K(X,Y)$.
\end{itemize}
\end{proposition}
\begin{proof}
	Since the covariant derivative of the $(1,1)$-tensor field $\phi$ is given by 
	\begin{align*}
		\forall X,Y\in \mathfrak{X}(M); (\nabla _X \phi)Y=\nabla _X\phi Y-\phi\nabla _XY,
	\end{align*}
	 and the relation 
\begin{align*}
	\forall X,Y\in \mathfrak{X}(M); (\nabla^{\circ} _X \phi)Y=(\nabla _X \phi)Y+2\phi K(X,Y)
\end{align*}	 
	  holds by Lemma \ref{5.6}, it is straightforward to verify that the three statements above are equivalent.
\end{proof}
\begin{proposition}
	If $(M, \phi, \xi, \eta, g, \nabla = \nabla^\circ + K)$ be an almost contact statistical manifold with $\phi$-compatible statistical connection, then we have $\nabla_X\xi \parallel\xi$ and $\nabla^{\circ}_X\xi \parallel\xi$.
\end{proposition}
\begin{proof}
	First, from Lemma \ref{4.3} we conclude that $K(X,\xi) \parallel \xi$. Then, due to one-dimensionality of $Ker(\eta)$ and the $\phi$-compatibility of the statistical connection $\nabla$, we can say:
	\begin{align*}
		K(X,\xi) \parallel \xi \; &\Longrightarrow (\nabla_X \xi - \nabla^{\circ}_X \xi) \parallel \xi\\
		&\Longrightarrow \phi(\nabla_X \xi - \nabla^{\circ}_X \xi)=0\\
		&\Longrightarrow \phi (\nabla ^\circ _X \xi)=0\\
		&\Longrightarrow \nabla^{\circ}_X \xi \parallel \xi \\
		&\Longrightarrow\nabla_X \xi \parallel \xi.
	\end{align*}	
\end{proof}
\begin{lemma} \label{6.3}
	If $(M, \phi, \xi, \eta, g, \nabla = \nabla^\circ + K)$ be an almost contact statistical manifold with $\phi$-compatible statistical connection, then 
	\begin{align*}
	\forall Y, Z \in \mathfrak{X}(M) ;(\nabla_X g)(Y, \phi Z) = -(\nabla_X g)(Z, \phi Y).
	\end{align*}
\end{lemma}
\begin{proof}
	A straightforward calculation, based on the covariant derivative of the Riemannian metric $g$ and the $\phi$-compatibility of the statistical connection $ \nabla $, leads to the desired result.
	\begin{align*}
		(\nabla_X g)(Y, \phi Z) &= 
		X \cdot g(Y, \phi Z) - g(\nabla_X Y, \phi Z) - g(Y, \nabla_X \phi Z) \\
		&= -X \cdot g(\phi Y, Z) + g(\phi \nabla_X Y, Z) - g(Y, \phi \nabla_X Z) \\
		&= -X \cdot g(\phi Y, Z) + g(\nabla_X \phi Y, Z) + g(\phi Y, \nabla_X Z) \\
		&= -(\nabla_X g)(\phi Y, Z) \\
		&= -(\nabla_X g)(Z, \phi Y).
	\end{align*}
\end{proof}
Now, we define the mapping \( \varPsi \) as  
\[
\varPsi : \mathfrak{X}(M) \longrightarrow A^2(M)
\]
\[
X \longmapsto \varPsi_X
\]
 where 
 \[
 \varPsi_X : \mathfrak{X}(M) \times \mathfrak{X}(M) \longrightarrow C^\infty(M)
 \]
 \[
\varPsi_X(Y, Z) := (\nabla_X g)(Y, \phi Z)
 \]
  holds. Note that, according to Lemma \ref{6.3}, and given that \( \nabla _Xg \) is a covariant tensor of rank 2 on \( M \), it follows that \(\varPsi_X \) defines a differential 2-form on \( M \). Moreover, it can be seen that \( \varPsi \) is an \( C^\infty (M) \)-module homomorphism.
  \begin{proposition}
   If \( M \) is an almost contact statistical manifold with a \(\phi\)-compatible statistical connection, then the mapping \( \varPsi \) has the following properties:
   \begin{itemize}
   	\item $\varPsi_X(Y, Z) = \varPsi_Y(X, Z) = \varPsi_Z(Y, X) $,
   	\item $\varPsi_X(\phi Y, Z) = -\varPsi_X(Y, \phi Z)$,
   	\item $\varPsi_X(\phi Y, \phi Z) = \varPsi_X(Y, Z)$.
   \end{itemize}
  	\end{proposition}
  	\begin{proposition} \label{6.5}
  	Let $(M, \phi, \xi, \eta, g, \nabla = \nabla^\circ + K)$ be an almost contact statistical manifold. If \( \nabla \) is a \( \phi \)-compatible statistical connection, then the following equation naturally arises:
  	\begin{align*}
  		\Psi_X(Y, Z) = 2g(\phi K(Y, Z), X).
  	\end{align*}
  	\begin{proof}
  \begin{align*}
  	\Psi_X(Y, Z) &= (\nabla_X g)(Y, \phi Z)  = X \cdot g(Y, \phi Z) - g(\nabla_X Y, \phi Z) - g(Y, \nabla_X \phi Z)\\
  &= X \cdot g(Y, \phi Z)  - g(\nabla^{\circ}_X Y + K(X, Y), \phi Z)  - g(Y, \nabla^{\circ}_X \phi Z + K(X, \phi Z))\\
 &= (\nabla^{\circ}_X g)(Y, \phi Z)  - 2g(K(X, Y), \phi Z)  =2g(\phi K(Y, Z), X) .  	
  \end{align*}
  	\end{proof}
  	\end{proposition}
  	\begin{proposition}
  		Let  $(M, \phi, \xi, \eta, g, \nabla = \nabla^\circ + K)$ be an almost contact statistical manifold with a $\phi$-compatible statistical connection. Then 		$\mathcal{K} _{\phi}=0$ if and only if $\varPsi =0$. 
  	\end{proposition}		
  	\begin{proof}	
  		If 	$\mathcal{K} _{\phi}=0$ is satisfied, then by Theorem \ref{5.8} and lemma \ref{4.5}, the relation $\phi K(Y,Z)=0$ holds true for any vector fields $Y$ and $Z$ on $M$.
  		Now, by Proposition \ref{6.5}, we clearly obtain that $\varPsi =0$.
  		
  		Conversely, if for any vector field $X\in \mathfrak{X}(M)$, the tensor $\varPsi_X =0$ vanishes, then by Proposition \ref{6.5} and the non-degeneracy of the Riemannian metric $g$, we conclude that 
  	\begin{align*}
  \forall Y,Z\in\mathfrak{X}(M);\; \phi K(Y,Z)=0,
  	\end{align*}
  		which is equivalent to $\mathcal{K}_{\phi}=0$.
  \end{proof}
  \begin{proposition}\label{6.7}
  	Let  $(M, \phi, \xi, \eta, g, \nabla = \nabla^\circ + K)$  be an almost contact statistical manifold. If $\nabla$ is a $\phi$-compatible statistical connection, then $\mathcal{K}_{\phi}=0$.
  \end{proposition}
  \begin{proof}
  	It is known that $\varPsi _X$ defines a differential 2-form on $M$ for any vector field $X$ .
  	On the other hand, by Proposition \ref{6.5}, it is clear that $\varPsi_X$ is also a symmetric tensor field on $M$.
  	Therefore, we easily conclude that the map $\varPsi$ , and equivalently the $\phi$-sectional $K$-curvature  $\mathcal{K}_{\phi}$, vanishes.
  \end{proof}
  \begin{theorem}\label{6.8}
  	If  $(M, \phi, \xi, \eta, g, \nabla = \nabla^\circ + K)$ is an almost contact statistical manifold with a $\phi$-compatible statistical connection, then $M$ is a cosymplectic statistical manifold with $\phi$-sectional $K$-curvature  $\mathcal{K}_{\phi}=0$.
  \end{theorem}
  \begin{proof}
  By Proposition \ref{5.6}, we know that in any almost contact statistical manifold, the relation
 \begin{align*}
 	(\nabla^\circ _X\phi)Y= (\nabla _X\phi)Y+2\phi K(X,Y)
 \end{align*}
  holds. On the other hand, by proposition \ref{6.7}, we know that if $\nabla$ is a $\phi$-compatible statistical connection, then  $\mathcal{K}_{\phi}=0$.
  Therefore, for any vector fields $X$ and $Y$ on $M$, we have $	(\nabla^\circ _X\phi)Y=0$,
  which shows that $M$ is a cosymplectic statistical manifold of constant $\phi$-sectional $K$-curvature zero.
  \end{proof}
\begin{corollary}
	As an immediate consequence, we conclude that Sasakian and Kenmotsu statistical manifolds cannot admit $\phi$-compatible statistical connections. This follows from the fact that neither Sasakian nor Kenmotsu manifolds are cosymplectic, and thus the condition
	$\nabla \phi =0$ is inherently incompatible with their structure.
\end{corollary}
\begin{corollary}
	Let $(M, \phi, \xi, \eta, g)$ be an almost contact statistical manifold, and suppose
	$\nabla = \nabla^\circ + K$ is a nontrivial $\phi$-compatible statistical structure on $M$. In this case, the vector field $\xi$ is geodesic with respect to the Levi-Civita connection $\nabla^{\circ}$, whereas it is not geodesic with respect to the statistical connection $\nabla$.
\end{corollary}
\begin{proof}
	Combining Theorem \ref{6.8} with Theorem \ref{5.11} yields the result in a straightforward manner.
\end{proof}

 \begin{example}	
 	Let $M=\mathbb{R}^{2n+1}$ and $\left(x^1, y^1,..., x^n, y^n, z\right)$ be the natural coordinate system on it. Suppose that
 		\begin{align*}
 	&g=(dx^1)^2+(dy^1)^2+...+(dx^n)^2+(dy^n)^2+(dz)^2;\\
 		&\phi(\frac{\partial}{\partial x^i})=\frac{\partial}{\partial y^i},\,\phi(\frac{\partial}{\partial y^i})=-\frac{\partial}{\partial x^i},\, \phi(\frac{\partial}{\partial z})=0;\\
 	&\xi =\frac{\partial}{\partial z};\quad \eta (X)=g(X,\xi).
 		\end{align*}
 	Then $(M,\phi,\xi, \eta,g)$ is an almost contact metric manifold.
 	
 	Moreover, consider the torsion-free affine connection on defined by:
 	\begin{align*}
 	&\nabla _{\frac{\partial}{\partial x^i}}{\frac{\partial}{\partial x^i}}=
 	\nabla _{\frac{\partial}{\partial x^i}}{\frac{\partial}{\partial y^i}}=
 	\nabla _{\frac{\partial}{\partial y^i}}{\frac{\partial}{\partial x^i}}=
 	\nabla _{\frac{\partial}{\partial y^i}}{\frac{\partial}{\partial y^i}}=0,\\
 	&\nabla _{\frac{\partial}{\partial x^i}}{\xi}=
 	\nabla _{\xi}{\frac{\partial}{\partial x^i}}=
 	\nabla _{\frac{\partial}{\partial y^i}}{\xi}=
 	\nabla _{x^i}{\frac{\partial}{\partial y^i}}=0,
 	\nabla _{\xi}{\xi}=\xi.
 \end{align*}
 	Then $(M,\phi,\xi, \eta,g, \nabla)$ is a cosymplectic statistical manifold of constant 
 	$\phi$-sectional $K$-curvature $\mathcal{K}_{\phi}=0$.
 	\begin{proof}
 	Based on the Riemannian metric defined on $M$, it is clear that
 	\begin{align*}
 		g(\frac{\partial}{\partial x^i},\frac{\partial}{\partial x^i})=
 		g(\frac{\partial}{\partial y^i},\frac{\partial}{\partial y^i})=
 		g(\xi,\xi)=1,\;
 		g(\frac{\partial}{\partial x^i},\frac{\partial}{\partial y^i})=
 		g(\frac{\partial}{\partial x^i},\xi)=
 		g(\frac{\partial}{\partial y^i},\xi)=0.	
    \end{align*}
 	Also, by straightforward computations, one can easily verify that
 $(M,\phi,\xi, \eta,g)$
 	is an almost contact statistical manifold.
 	
 	Since the coefficients of the Riemannian metric $g$ on $M$ are constant, the Christoffel symbols and the curvature tensor associated with the Levi-Civita connection vanish.
 	Therefore, if $\nabla^\circ$ denotes the unique Levi-Civita connection of the manifold , then
 	 \begin{align*}
 		\nabla^{\circ} _{\frac{\partial}{\partial x^i}}{\frac{\partial}{\partial x^i}} =
 		 \nabla^{\circ} _{\frac{\partial}{\partial x^i}}\frac{\partial}{\partial y^i}=
 		 \nabla^{\circ} _{\frac{\partial}{\partial x^i}}\xi =
 		  \nabla^{\circ} _{\frac{\partial}{\partial y^i}}\frac{\partial}{\partial y^i}=
 		  \nabla^{\circ} _{\frac{\partial}{\partial y^i}} \xi=\nabla^{\circ}_{\xi}\xi=0.
 	\end{align*}
 	As the space of all torsion-free affine connections on $M$ is an affine space whose associated vector space is
 	$S^2(M;TM)$,
 	the difference tensor
 	$K=\nabla - \nabla^0\in\otimes ^1_2 M$
 	is obtained as
 	\begin{align*}
 		K(\frac{\partial}{\partial x^i},\frac{\partial}{\partial x^i})=
 		K(\frac{\partial}{\partial x^i},\frac{\partial}{\partial y^i})=
 		K(\frac{\partial}{\partial x^i},\xi)=
     	K(\frac{\partial}{\partial y^i},\frac{\partial}{\partial y^i})=
 		K(\frac{\partial}{\partial y^i},\xi)=0,\;
 		K(\xi,\xi)=\xi.
 	\end{align*}
 	It is clearly seen that
 		$\nabla= K + \nabla^0$
 	is a $\phi$-compatible statistical connection on $(M, \phi, \xi, \eta, g)$.
 	Hence, by Theorem \ref{6.8}, we conclude that
 	$(M, \phi, \xi, \eta, g, \nabla = \nabla^\circ + K)$
 	is a cosymplectic statistical manifold with $\mathcal{K}_{\phi}=0$.
 \end{proof}	
\end{example} 
 In the following example, we have directly used an example from reference \cite{18} to support our discussion.
 \begin{example}
 	Let $M=\mathbb{R}^3$ and $\{x,y,z\}$ be the natural coordinate system. Suppose that
 	\begin{align*}
 		g=(dx)^2+(dy)^2+(dz)^2;\:\phi\frac{\partial}{\partial x}=\frac{\partial}{\partial y},\,\phi\frac{\partial}{\partial y}=-\frac{\partial}{\partial x},\, \phi\frac{\partial}{\partial z}=0;\:	\xi =\frac{\partial}{\partial z};\: \eta (X)=g(X,\xi).
 	\end{align*}
 	Then $(M,\phi,\xi, \eta,g)$ is a cosymplectic manifold.
 	Moreover, let us consider the torsion-free affine connection $\nabla$ on $M$, defined by the following relations:
 \begin{align*}
 	&\nabla _{\frac{\partial}{\partial x}}{\frac{\partial}{\partial x}}=-\frac{1}{2}\frac{\partial}{\partial x}+\frac{1}{2}\frac{\partial}{\partial y},\quad
 	\nabla _{\frac{\partial}{\partial x}}{\frac{\partial}{\partial y}}=\frac{1}{2}\frac{\partial}{\partial x}+\frac{1}{2}\frac{\partial}{\partial y},\\
	&\nabla _{\frac{\partial}{\partial y}}{\frac{\partial}{\partial x}}=\frac{1}{2}\frac{\partial}{\partial x}+\frac{1}{2}\frac{\partial}{\partial y},\quad
 	\nabla _{\frac{\partial}{\partial y}}{\frac{\partial}{\partial y}}=\frac{1}{2}\frac{\partial}{\partial x}-\frac{1}{2}\frac{\partial}{\partial y},\\
 &\nabla _{\frac{\partial}{\partial x}}\xi = \nabla _{\xi}\frac{\partial}{\partial x}=\nabla _{\frac{\partial}{\partial y}}\xi = \nabla _{\xi}\frac{\partial}{\partial y}=\nabla _{\xi} \xi=0.
 \end{align*}
 	Then $(M,\phi ,\xi , \eta,g,\nabla)$ is a cosymplectic statistical manifold with $\mathcal{K}_{\phi}=-1$.\\
 	This provides an example of a cosymplectic statistical manifold in which the connection $\nabla$ is not a $\phi$-compatible statistical connection.
 \begin{proof}	
 	First, considering the Riemannian metric $g$ defined on $M$, it is clear that
 \begin{align*}
 	g(\frac{\partial}{\partial x},\frac{\partial}{\partial x})=g(\frac{\partial}{\partial y},\frac{\partial}{\partial y})=g(\xi,\xi)=1,\;
 		g(\frac{\partial}{\partial x},\frac{\partial}{\partial y})=g(\frac{\partial}{\partial x},\xi)=g(\frac{\partial}{\partial y},\xi)=0.
 \end{align*}
 	Furthermore, straightforward calculations easily show that
 $(M, \phi, \xi, \eta, g)$ 	is an almost contact metric manifold. Since the coefficients of the Riemannian metric $g$ on $M$ are constant, the Christoffel symbols and the curvature tensor associated with the Levi-Civita connection vanish. Thus, if $\nabla^{\circ}$ is the unique Levi-Civita connection of the manifold $M$, then
 \begin{align*}
 	\nabla^{\circ} _{\frac{\partial}{\partial x}}{\frac{\partial}{\partial x}} =
 	 \nabla^{\circ} _{\frac{\partial}{\partial x}}\frac{\partial}{\partial y}=
 	 \nabla^{\circ} _{\frac{\partial}{\partial x}}\xi =
 	  \nabla^{\circ} _{\frac{\partial}{\partial y}}\frac{\partial}{\partial y}=
 	  \nabla^{\circ} _{\frac{\partial}{\partial y}} \xi=
 	  \nabla^{\circ}_{\xi}\xi=0.
 \end{align*}
 	Therefore, the condition $\nabla ^{\circ}\phi$ is equal to zero, and we can conclude that  
 	$(M, \phi, \xi, \eta, g)$ is a cosymplectic manifold. Now, considering the difference tensor
 	$K=\nabla - \nabla ^{\circ}$ as a symmetric covariant 2-tensor with values in the tangent bundle $TM$, it follows that $(g,\nabla =\nabla ^{\circ}+K)$ provides a statistical structure on
 	 $(M, \phi, \xi, \eta, g)$ 	that satisfies the definition \ref{4.1}.
 	Moreover, the theorem \ref{6.8} states that in the cosymplectic statistical manifold
 	 $(M, \phi, \xi, \eta, g,\nabla =\nabla ^{\circ}+K)$ the value of $\phi$-sectional $K$-curvature $\mathcal{K}_{\phi}$ is negative. Now we demonstrate that the value of $\mathcal{K}_{\phi}$ is equal to the constant -1. To do so, according to theorem \ref{5.3},	we calculate $\left\| K(X, X) \right\|^2$ and $\|X\|^4$.
 	 \begin{align*}
 	K(X, X) &= f_1^2 K(\frac{\partial}{\partial x}, \frac{\partial}{\partial x})
 	+ 2 f_1 f_2 K(\frac{\partial}{\partial x}, \frac{\partial}{\partial y})
 	+ f_2^2 K(\frac{\partial}{\partial y}, \frac{\partial}{\partial y}) \\
 	&= ( -\frac{1}{2} f_1^2 + f_1 f_2 + \frac{1}{2} f_2^2 ) \frac{\partial}{\partial x}
 	+ ( \frac{1}{2} f_1^2 + f_1 f_2 - \frac{1}{2} f_2^2 ) \frac{\partial}{\partial y},\\
 		\left\| K(X, X) \right\|^2 &= ( -\frac{1}{2} f_1^2 + f_1 f_2 + \frac{1}{2} f_2^2 )^2 +
 		( \frac{1}{2} f_1^2 + f_1 f_2 - \frac{1}{2} f_2^2 )^2 
 	= \frac{1}{2} \left( f_1^2 + f_2^2 \right)^2,\\
 	\|X\|^4 = g(X, X)^2 
 	&= \left( f_1^2 g(\frac{\partial}{\partial x}, \frac{\partial}{\partial x})
 	+ 2 f_1 f_2 g(\frac{\partial}{\partial x}, \frac{\partial}{\partial y})
 	+ f_2^2 g(\frac{\partial}{\partial y}, \frac{\partial}{\partial y}) \right)^2 
 	= \left( f_1^2 + f_2^2 \right)^2.
 \end{align*}
 	Therefore, we can conclude that
 \[
\mathcal{K}_\phi(X) = -2 \frac{\left\| K(X, X) \right\|^2}{\|X\|^4} = -1
 \]
 	holds.
 	\end{proof}
 \end{example}

\end{document}